\documentclass[10pt, a4paper]{amsart}
\usepackage{geometry} 
\usepackage{url}
\usepackage{palatino}
 
\usepackage{graphicx}
\usepackage[colorlinks, citecolor=blue, linkcolor=black]{hyperref}

 \usepackage{float}

\usepackage{stmaryrd}
\usepackage{amsthm}
\usepackage{amssymb}
\usepackage{latexsym}
\usepackage{amsfonts}
\usepackage{amsmath}
\usepackage{amstext}
\usepackage{accents}
\usepackage{cancel}
\usepackage{color}
\usepackage{soul}
\usepackage{enumerate}
\usepackage{eucal}
\usepackage{amscd}
\usepackage{hyperref}
\usepackage{graphicx}
\usepackage{verbatim}
\usepackage{caption}
\usepackage{enumitem}

\usepackage[usenames]{xcolor}
\setstcolor{red}

\definecolor{greenish}{RGB}{50,160,0}

\newtheorem{theorem}{Theorem}[section]

\newtheorem{lemma}[theorem]{Lemma}
\newtheorem{proposition}[theorem]{Proposition}

\theoremstyle{definition}

\newtheorem{example}[theorem]{Example}

\theoremstyle{remark}
\newtheorem{remark}[theorem]{Remark}

\numberwithin{equation}{section}

 \makeatletter
\newcommand{\fixed@sra}{$\vrule height 2\fontdimen22\textfont2 width 0pt\shortrightarrow$}
\newcommand{\shortarrow}[1]{
  \mathrel{\text{\rotatebox[origin=c]{\numexpr#1*45}{\fixed@sra}}}
}
\makeatother

\begin{document}
\title[{{\textsc{Uniqueness of minimal surfaces, Jacobi fields, and flat structures}}}]{Uniqueness of minimal surfaces, Jacobi fields, and flat structures} 
  
 \author[Hojoo Lee]{Hojoo Lee}
\address{Department of Mathematical Sciences, Seoul National University, Gwan Ak Ro 1, Gwanak Gu, Seoul 08826, Korea}
\email{momentmaplee@gmail.com}

\begin{abstract}
Inspired by the Finn-Osserman (1964), Chern (1969), do Carmo-Peng (1979) proofs of the Bernstein theorem, which 
characterizes flat planes as the only entire minimal graphs in ${\mathbb{R}}^{3}$, we 
prove a new rigidity theorem for associate families connecting the doubly periodic Scherk graphs and the singly periodic Scherk towers.
Our characterization of Scherk's surfaces discovers a new idea from the original Finn-Osserman curvature estimate.
Combining two generically independent flat structures introduced by Chern and Ricci, we shall construct geometric harmonic functions on 
minimal surfaces, and establish that periodic minimal surfaces admit fresh uniqueness results. 
\end{abstract}

 
 \maketitle

  \bigskip

\begin{center}
Dedicated to Jaigyoung Choe on the occasion of his $65$th birthday
\end{center}

  \bigskip

   \tableofcontents

\section{Ubiquitous flat structures}    

Flat structures are one of the most powerful tools in the  theory of minimal surfaces. J. Choe and 
R. Schoen \cite{ChoeSchoen2009} employed flat 
structures to investigate new isoperimetric inequalities. Flat structures play an essential role in number of proofs of the existence of minimal surfaces. 
For instance, see \cite{WeberWolf2002} by M. Weber and M. Wolf,  \cite{Huff2006} by R. Huff, and \cite{Douglas2014} by C. Douglas.
In particular, C. Douglas \cite{Douglas2014} proved the existence of a one parameter family of doubly periodic minimal surfaces,
which associates a singly periodic genus one helicoid as their limit surfaces.

To present a non-complex-analytic proof of the Bernstein theorem that 
entire minimal graphs should be flat, S.-S. Chern \cite{Chern1969} constructed flat structures using angle functions, which are the Jacobi fields 
induced by translations. Moreover, M. do Carmo and C. K. Peng  \cite[p. 905]{CP1979} also used Ricci's flat structure \cite[p. 124]{Bl1950} 
to prove generalized Bernstein's theorem that the complete stable minimal surfaces should be flat. The flat structures appeared 
in these applications could be explicitly expressed in terms of the holomoprhic data on minimal surfaces.

 Our main goal of this paper is to provide new rigidity results for minimal surfaces in terms of Jacobi fields and harmonic functions induced by two generically independent flat structures due to Chern and Ricci. Our characterization of Scherk's minimal surfaces discovers a new idea from
 the original Finn-Osserman curvature estimate for minimal surfaces.

    \newpage

 \section{Three Jacobi fields on Scherk's surfaces $\&$ Finn-Osserman curvature estimate}

R. Finn and  R. Osserman \cite{FO1964} presented a brilliant proof of Bernstein's beautiful theorem that the flat planes are the only \emph{entire} 
minimal graphs in  ${\mathbb{R}}^{3}$. To deduce their optimal curvature estimate result  \cite[p. 362]{FO1964} for arbitrary minimal 
graphs defined over round discs, they used a curvature estimate for the doubly periodic Scherk's graph defined over the white 
squares of a checkerboard  ${\mathbb{R}}^{2}$. Rescaling the Gauss curvature at the origin as one, the fundamental piece of  the doubly periodic 
 Scherk's surface can be represented as the graph
 \[
   z = {\mathcal{F}} \left(x, y \right) =  \ln \left( \frac{\cos y}{\cos x} \right) = \ln \left(  \cos y  \right) -  \ln \left(  \cos x  \right), \quad    
     - \frac{\pi}{2} <x, y< \frac{\pi}{2}, 
 \] 
which is oriented by the up-ward pointing unit normal vector field 
{\small{
\[
  {\mathbf{n}}= 
 \begin{bmatrix}
    {\mathbf{n}}_{1}   \\
    {\mathbf{n}}_{2}   \\
    {\mathbf{n}}_{3} 
 \end{bmatrix}
 = \frac{1}{\,\sqrt{\,1+ {{\mathcal{F}}_{x}}^{2} + {{\mathcal{F}}_{y}}^{2} \,}\,} 
 \begin{bmatrix}
  - {\mathcal{F}}_{x}  \\
 - {\mathcal{F}}_{y}  \\
  1
 \end{bmatrix}
  = \frac{1}{\,\sqrt{\,1+ {\tan}^{2}x +  {\tan}^{2}y \,}\,} 
 \begin{bmatrix}
  - \tan x  \\
  \tan y  \\
  1
 \end{bmatrix}.
\]
}}

The starting point of the Finn-Osserman proof is a neat observation that, on the Scherk's doubly periodic graph, 
Gauss curvature $\mathcal{K}$ can be explicitly estimated by a function of the Jacobi field ${\mathbf{n}}_{3}= {\langle {\mathbf{n}}, {\varepsilon}_{3} \rangle}_{{\mathbb{R}}^{3}}$ induced by the vertical unit vector field  ${\varepsilon}_{3}={\left(0, 0, 1 \right)}^{\intercal}$. Indeed, they observe that, on Scherk's doubly periodic graph $z = {\mathcal{F}} \left(x, y \right)$, the Gauss curvature function 
{\small{
\[
   \mathcal{K} =\mathcal{K}\left(x, y \right)
=   \frac{ {\mathcal{F}}_{xx}  {\mathcal{F}}_{yy} - { {\mathcal{F}}_{xy}}^{2}}{\,{\left(1+ { {\mathcal{F}}_{x}}^{2} + { {\mathcal{F}}_{y}}^{2} \right)}^{2}\,} 
=  -  \frac{ \, \left( 1 + {\tan}^{2} x   \right)  \left( 1 + {\tan}^{2} y \right) \,}{\,{\left( 1 + {\tan}^{2} x +  {\tan}^{2} y \right)}^{2}\,} <0
\]
}}
 obeys a simple but neat  \emph{curvature estimate} \cite[p. 359, Equation (21)]{FO1964} 
\begin{equation} \label{FO estimate}
  \sqrt{\, - \mathcal{K} \,} \leq   \frac{\,  1 + {{\mathbf{n}}_{3}}^{2}  }{2}.
\end{equation}

It would be not possible to write the curvature $\mathcal{K}$ sorely in terms of the \emph{single} Jacobi field 
\[
{\mathbf{n}}_{3}= {\langle {\mathbf{n}}, {\varepsilon}_{3} \rangle}_{{\mathbb{R}}^{3}}
\]
induced by the \emph{vertical} vector field 
${\varepsilon}_{3}={\left(0, 0, 1 \right)}^{\intercal}$. What was not unravelled in 
 Finn-Osserman's original paper \cite{FO1964} is our new observation that, on the Scherk graph,   
its Gauss curvature $\mathcal{K}$ could be explicitly expressed as a function of \emph{two} Jacobi fields 
\[
{\mathbf{n}}_{1} =  {\langle {\mathbf{n}}, {\varepsilon}_{1} \rangle}_{{\mathbb{R}}^{3}} \;\, \text{and} \;\, {\mathbf{n}}_{2}=  {\langle {\mathbf{n}}, {\varepsilon}_{2} \rangle}_{{\mathbb{R}}^{3}}
\]
induced by \emph{two orthogonal horizontal} vector fields ${\varepsilon}_{1}={\left(1, 0, 0 \right)}^{\intercal}$ and ${\varepsilon}_{2}={\left(0, 1, 0 \right)}^{\intercal}$, respectively. 

Indeed, on the  Scherk graph  $z= {\mathcal{F}} \left(x, y \right)=  \ln \left(  \cos y  \right) -  \ln \left(  \cos x  \right)$, which was normalized so that its Gauss curvature at the origin is equal to $1$, we discover a new \emph{curvature identity} 
\begin{equation} \label{FO estimate twisted}
  - \mathcal{K}   = \left( 1 - {{\mathbf{n}}_{1}}^{2} \right) \left( 1 - {{\mathbf{n}}_{2}}^{2} \right). 
\end{equation}
Moreover, the curvature identity (\ref{FO estimate twisted}) immediately implies Finn-Osserman's curvature estimate:
\[
 \sqrt{  \, - \mathcal{K} \,} = \sqrt{\, \left( 1 - {{\mathbf{n}}_{1}}^{2} \right) \left( 1 - {{\mathbf{n}}_{2}}^{2} \right) \,} \leq \frac{\, \left( 1 - {{\mathbf{n}}_{1}}^{2} \right) + \left( 1 - {{\mathbf{n}}_{2}}^{2} \right)  \, }{2}    =   \frac{\,  1 + {{\mathbf{n}}_{3}}^{2}  }{2}.
\]

We shall prove that the geometric equality (\ref{FO estimate twisted}) captures the uniqueness of Scherk's surfaces, up to associate families, and present  
 periodic minimal surfaces, which also admit rigidity results in terms of Gauss curvature and finitely many Jacobi fields induced by ambient translations.

  \begin{figure}[H]
 \centering
 \includegraphics[height=4.20cm]{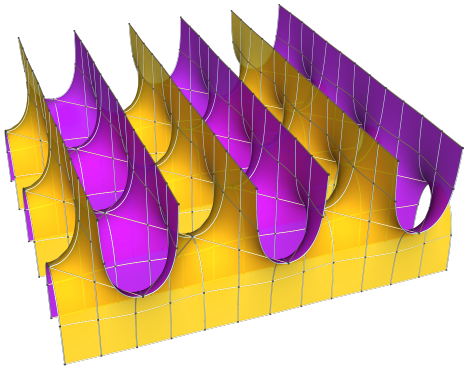} \quad \quad 
  \includegraphics[height=4.20cm]{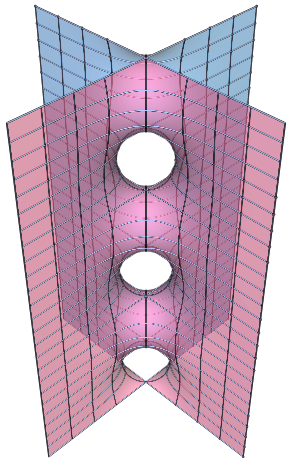}
 \caption{\small{Approximations \cite{WebEnn} of parts of Scherk's periodic minimal surfaces}} 
 \end{figure}

\begin{theorem}[\textbf{Rigidity of Scherk's surfaces in terms of Gauss curvature and two Jacobi fields}] \label{Scherk Uniqueness}
If a negatively curved minimal surface with Gauss curvature $\mathcal{K}$ and the unit normal vector field $\mathbf{n}$ admits two Jacobi fields $ {{\mathbf{n}}_{{}_{{\mathbf{V}}_{1}}}} =  {\langle {\mathbf{n}}, {\mathbf{V}}_{1}\rangle}_{{\mathbb{R}}^{3}}$ and 
${{\mathbf{n}}_{{}_{{\mathbf{V}}_{2}}}} = {\langle {\mathbf{n}}, {\mathbf{V}}_{2}\rangle}_{{\mathbb{R}}^{3}}$, induced by two constant \textbf{orthogonal} unit vector fields ${\mathbf{V}}_{1}$ and ${\mathbf{V}}_{2}$ in  ${\mathbb{R}}^{3}$, such that the Chern-Ricci harmonic function 
\begin{equation} \label{Scherk two Jacobi}
   \ln \frac{\, \left(  1 - {{\mathbf{n}}_{{}_{{\mathbf{V}}_{1}}}}^{2} \right) \left( 1 - {{\mathbf{n}}_{{}_{{\mathbf{V}}_{2}}}}^{2} \right)  \,}{ - \mathcal{K} }  
\end{equation}
is constant, then the minimal surface should be congruent to (a part of) a member of the associate family connecting a doubly periodic Scherk graph and a singly periodic Scherk tower, up to homotheties. 
\end{theorem}

For an alternative geometric interpretation of the Chern-Ricci harmonic function (\ref{Scherk two Jacobi}) in terms of the flat structures on minimal surfaces, see Lemma \ref{CR def} and Remark \ref{StwoFS}.

 \section{Two non-complex-analytic proofs of Bernstein's theorem}
 
It is natural to ask whether or not the function (\ref{Scherk two Jacobi}) in Theorem \ref{Scherk Uniqueness} admit natural \emph{geometric} meaning for 
arbitrary minimal surfaces. The answer is yes. Combining the key ingredients in two completely different proofs of Bernstein's theorem given by Chern  \cite{Chern1969} and do Carmo-Peng  \cite{CP1979} answers the reason why the functions in the linear combination (\ref{Scherk two Jacobi}) should 
be \emph{harmonic} on the negatively curved minimal surfaces. Their non-complex-analytic proofs, not exploiting the Enneper-Weierstrass reprsentation directly, require two \emph{generically independent} flat structures.
 
\begin{proposition}[\textbf{Harmonic functions induced by two flat structures}] \label{two flat}
 Let  $\Sigma$ be a minimal surface in ${\mathbb{R}}^{3}$ with the Gauss curvature function ${\mathcal{K}}={\mathcal{K}}_{ {\mathbf{g}}_{{}_{\Sigma}} } 
 \leq 0$ and the unit normal vector field $\mathbf{n}$ on  $\Sigma$. Given a constant unit vector field ${\mathbf{V}}$ in ${\mathbb{R}}^{3}$, we introduce the Jacobi field $\mathbf{n}_{{}_{\mathbf{V}}}:\Sigma \to \left[-1, 1\right]$ defined by 
\[
  {\mathbf{n}}_{{}_{\mathbf{V}}}= {\langle \mathbf{n}, {\mathbf{V}} \rangle}_{{\mathbb{R}}^{3}}, \quad p \in \Sigma.
\]
 It is called the angle function induced by the translation generated by the vector field ${\mathbf{V}}$ in the literature.
\begin{enumerate}
\item[\textbf{(a)}] \textbf{Chern's flat structure \cite[p. 53, Equation (2)]{Chern1969}}. The conformally changed metric 
\begin{equation} \label{Chern flat}
 {\mathbf{g}}_{{}_{\textrm{Chern}}} ={\left( 1 +   \mathbf{n}_{{}_{\mathbf{V}}} \right)}^{2}  {\mathbf{g}}_{{}_{\Sigma}}
\end{equation}
is flat on the points where $\mathbf{n}_{{}_{\mathbf{V}}}>-1$.
\item[\textbf{(b1)}] \textbf{Ricci's flat structure \cite[p. 124]{Bl1950}}. It was used in the do Carmo-Peng proof \cite[p. 905]{CP1979} of generalized Bernstein's theorem that the complete stable minimal surfaces are flat. The metric 
\begin{equation} \label{Ricci flat}
 {\mathbf{g}}_{{}_{\textrm{Ricci}}} = \sqrt{- {\mathcal{K}}_{ {\mathbf{g}}_{{}_{\Sigma}}} } \, {\mathbf{g}}_{{}_{\Sigma}}
\end{equation}
is flat on the points where ${\mathcal{K}}_{ {\mathbf{g}}_{{}_{\Sigma}}}<0$. 
\item[\textbf{(b2)}] Ricci \cite{Bl1950, Law1971} discovered that that every metric satisfying Ricci condition \ref{Ricci flat} can be realized on a negatively curved minimal surfaces in ${\mathbb{R}}^{3}$. See Lawson's counterexample \cite[Remark 12.1]{Law1970} for the metics having a \emph{flat} point. 
\item[\textbf{(c)}] \textbf{Chern-Ricci harmonic functions} \cite[Theorem 2.1]{Lee2017}. \label{curvature cc} As observed in \cite[Theorem 2.1]{Lee2017} and the two-page expository note \cite{Lee2018}, these two flat structures induce harmonic functions. Indeed, 
whenever ${\mathcal{K}}<0$ and $\mathbf{n}_{{}_{\mathbf{V}}}>-1$ on the minimal surface, since two conformally changed metrics 
\[
\widetilde{g} = {\mathbf{g}}_{{}_{\textrm{Ricci}}}   \quad \text{and} \quad 
{\mathbf{g}}_{{}_{\textrm{Chern}}}= \star \widetilde{g}  =  \frac{  { \left( 1 + \mathbf{n}_{{}_{\mathbf{V}}} \right)}^{2}  }{  \sqrt{  \,- \mathcal{K} \, }  }   {\mathbf{g}}_{{}_{\textrm{Ricci}}} 
\]
are \emph{flat}, the classical curvature formula for conformally changed metrics 
 \begin{equation} \label{curvature cc}
 {\mathcal{K}}_{ \star \widetilde{g} } = \frac{1}{\star}   {\mathcal{K}}_{  \widetilde{g} }  - \frac{1}{2 \star} {\triangle}_{ \widetilde{g} } \ln \star 
 \end{equation}
 immediately implies that the induced logarithmic quotient function
\begin{equation} \label{CR harmonic}
   \ln  \star  =   \ln    \frac{  { \left( 1 + \mathbf{n}_{{}_{\mathbf{V}}} \right)}^{2}  }{  \sqrt{  \,- \mathcal{K} \, }  }  
\end{equation}
should be harmonic with respect to metrics ${\mathbf{g}}_{{}_{\textrm{Ricci}}} $, ${\mathbf{g}}_{{}_{\textrm{Chern}}}$, and $ {\mathbf{g}_{{}_{\Sigma}}}$.
 We call it the \textbf{Chern-Ricci harmonic function} \cite[Theorem 2.1]{Lee2017} with respect to the unit vector field ${\mathbf{V}}$.
\end{enumerate}
\end{proposition}
 
 \begin{remark}[\textbf{Rigidity of Enneper's algebraic minimal surfaces in terms of two flat structures}]
 It is natural to ask whether or not the two flat metrics ${\mathbf{g}}_{{}_{\textrm{Chern}}} ={\left( 1 +   \mathbf{n}_{{}_{\mathbf{V}}} \right)}^{2}  {\mathbf{g}}_{{}_{\Sigma}}$ and ${\mathbf{g}}_{{}_{\textrm{Ricci}}} = \sqrt{- {\mathcal{K}}_{ {\mathbf{g}}_{{}_{\Sigma}}} } \, {\mathbf{g}}_{{}_{\Sigma}}$ are indeed generically independent. The answer is essentially yes. There exists one exception.  \cite[Theorem 3.1]{Lee2017} shows that, on a negatively curved minimal surface, the geometric equality
 \[
      {\mathbf{g}}_{{}_{\textrm{Chern}}} = c \,  {\mathbf{g}}_{{}_{\textrm{Ricci}}}
 \]
holds for some constant $c>0$ if and only if it becomes the Enneper surface. It is well-known that each member of the associate family of an Enneper surface is congruent to itself. The key point of this paper is to provide new
 rigidity results, which generalize \cite[Theorem 3.1]{Lee2017}. For instance, Theorem \ref{Scherk Uniqueness} captures the uniqueness of Scherk's surfaces up
 to associate families.
\end{remark}

 \begin{figure}[H]
 \centering
 \includegraphics[height=4.20cm]{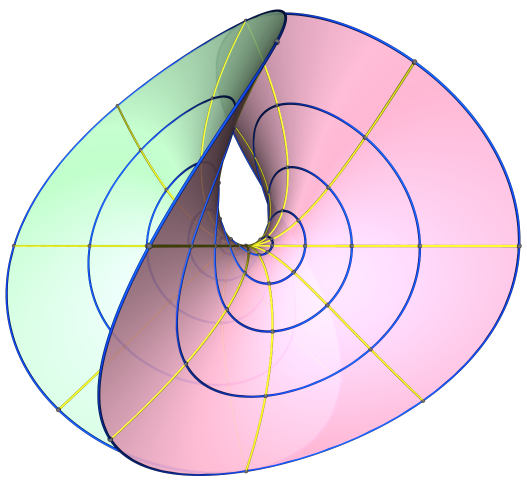}
 \caption{\small{An approximation \cite{WebEnn} of a part of Enneper's surface with total curvature $-4 \pi$}} 
 \end{figure}

 \section{Harmonic functions induced by Chern-Ricci flat structures}

Two coordinates ${\mathbf{x}}_{1}$ and ${\mathbf{x}}_{2}$ in the flat plane ${\mathbb{R}}^{2}$ are harmonic with respect to the flat metric ${d  {\mathbf{x}}_{1} }^{2} +  {d  {\mathbf{x}}_{2} }^{2} $. More generally, linear coordinates in ${\mathbb{R}}^{3}$ induce harmonic functions on minimal surfaces. The minimal surface theory and complex analysis are intertwined by the Enneper-Weierstrass representation. For instance, see \cite[p. 4]{Kar2003}, \cite[Lemma 8.2]{Oss86}, 
and \cite[Section 3.3]{Weber2001}. 

In Euclidean space $\left({\mathbb{R}}^{3}, {d  {\mathbf{x}}_{1} }^{2} +  {d  {\mathbf{x}}_{2} }^{2}  +  {d  {\mathbf{x}}_{3} }^{2} \right)$, a non-planar minimal surface $\Sigma$ is determined by the Weierstrass data $\left({\mathbf{G}}, d\mathbf{h}  \right)$, which encodes their geometric information. The meromorphic function ${\mathbf{G}}$ is obtained by applying the stereographic projection of the Gauss map. 
The height differential  $d \mathbf{h}$ is the holomorphic extension of the differential $d {\mathbf{x}}_{3}$. The 
Enneper-Weierstrass representation guarantees that the conformal harmonic patches can be constructed by solving the global integration problem
{\small{
\[ 
 \begin{bmatrix}
  d  {\mathbf{x}}_{1}   \\
  d  {\mathbf{x}}_{2}   \\
  d {\mathbf{x}}_{3} 
 \end{bmatrix}
 =
  \begin{bmatrix}
 \;   \textrm{Re} \left(  \;  \frac{\,1\,}{2} \left(  \frac{1}{{\mathbf{G}}} - {\mathbf{G}} \right) d\mathbf{h}  \;  \right)  \;  \\
 \;  \textrm{Re} \left(  \;  \frac{\,i\,}{2} \left(  \frac{1}{{\mathbf{G}}} + {\mathbf{G}} \right) d\mathbf{h}   \;  \right)  \;  \\
 \;  \textrm{Re} \left(  \;  d \mathbf{h}  \;  \right)  \; 
  \end{bmatrix}.
 \]
}}
 The associate family $\left\{{\Sigma}_{\theta} \right\}_{\theta \in \left[0, \frac{\pi}{2} \right]}$ of the minimal surface $\Sigma={\Sigma}_{0}$ can be obtained by integrating 
{\small{
\[ 
 \begin{bmatrix}
  d  {\left( {\mathbf{x}}_{\theta} \right)}_{1}   \\
  d  {\left( {\mathbf{x}}_{\theta} \right)}_{2}   \\
  d {\left( {\mathbf{x}}_{\theta} \right)}_{3}
 \end{bmatrix}
 =
  \begin{bmatrix}
 \;   \textrm{Re} \left(  \;  \frac{\,1\,}{2} \left(  \frac{1}{{\mathbf{G}}} - {\mathbf{G}} \right)  e^{i \theta} d\mathbf{h}  \;  \right)  \;  \\
 \;  \textrm{Re} \left(  \;  \frac{\,i\,}{2} \left(  \frac{1}{{\mathbf{G}}} + {\mathbf{G}} \right) e^{i \theta} d\mathbf{h}   \;  \right)  \;  \\
 \;  \textrm{Re} \left(  \;  e^{i \theta} d \mathbf{h}  \;  \right)  \; 
  \end{bmatrix}, \quad \theta \in \left[0, \frac{\pi}{2} \right].
 \]
}}
 In particular, ${\Sigma}_{\frac{\pi}{2}}$ is called the conjugate minimal surface of $\Sigma={\Sigma}_{0}$. Schwarz proved that any minimal surfaces locally isometric to $\Sigma$ should be congruent to a member of its associate family. We briefly sketch explicit examples of classical minimal surfaces with their Weierstrass data.

\bigskip

\begin{example}[\textbf{Complex analytic construction of minimal surfaces}] We take the meromorphic Gauss map ${\mathbf{G}}$ as the local conformal coordinates $\left( \, \text{Re} \left( {\mathbf{G}} \right), \, \text{Im} \left( {\mathbf{G}} \right)\, \right)$ on the minimal surface. 
\end{example}
\begin{enumerate}
  \item[]
\item[\textbf{(a)}] Enneper's surface with the Weierstrass data $\frac{1}{\,{\mathbf{G}}\,} d\mathbf{h}= d{\mathbf{G}}$ has the total curvature $-4\pi$. 
 For Enneper surfaces, the isometric deformation by the associate family induce rotations.
   \item[]
\item[\textbf{(b1)}] The Weierstrass data $\frac{1}{\,{\mathbf{G}}\,} d\mathbf{h}= \frac{1}{\, {\mathbf{G}}^{2}\,} d{\mathbf{G}}$ gives the catenoid with 
the total curvature $-4\pi$. R. Schoen \cite{Schoen1983}  proved a rigidity result that the catenoids are the only complete minimal surface embedded with finite total curvature and two ends. F. L\'{o}pez and A. Ros \cite{LO1991} established that planes and catenoids are the only properly embedded minimal surfaces with finite total curvature and genus zero.  
 \begin{figure}[H]
 \centering
 \includegraphics[height=4.20cm]{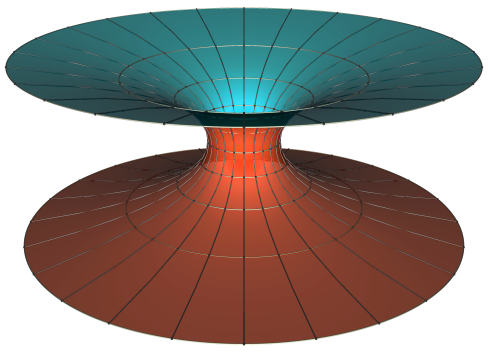} \quad \quad  \includegraphics[height=4.20cm]{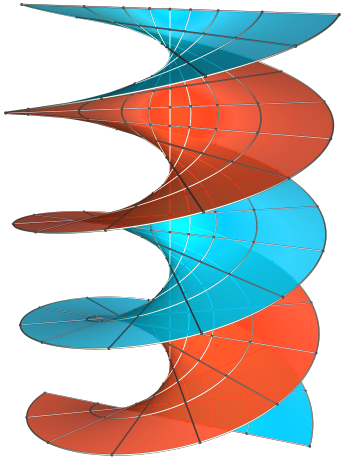}
 \caption{\small{Approximations \cite{WebEnn} of parts of a catenoid and a helicoid}} 
 \end{figure}
\item[\textbf{(b2)}] The catenoid is also induced by the Weierstrass data $\frac{1}{\,{\mathbf{G}}\,}  d\mathbf{h}= \frac{1}{\, {\left( {\mathbf{G}}^{2} -1 \right) }^{2} \,} d{\mathbf{G}}$. It recovers the data of the Jorge-Meeks $k$-noid \cite[Section 5]{JM1983} with $k=2$. See 
also \cite[Section 4.4]{Weber2001}. 
  \item[]
\item[\textbf{(c)}] We obtain helicoids by taking conjugate minimal surfaces of catenoids. The uniqueness of the associate families of 
catenoids and helicoids in terms of a Chern-Ricci harmonic function is proved in \cite[Theorem 3.3]{Lee2017}.
\end{enumerate}

  \begin{lemma}[\textbf{Chern-Ricci harmonic functions and generalization of the Chern metric}]  \label{CR def} Given a minimal surface $\Sigma$   
with the Gauss curvature $\mathcal{K}$  and the Weierstrass data $\left({\mathbf{G}}, d\mathbf{h}  \right)$ and a finitely many distinct constant 
unit vector fields
${\mathbf{V}}_{1}$, $\cdots$, ${\mathbf{V}}_{N}$ and weights ${\lambda}_{1}$, $\cdots$,  $\lambda_{N} \in \mathbb{Q}$ with ${\lambda}_{1} + 
\cdots + {\lambda}_{N}=1$, we introduce the Chern-Ricci function by the linear combination
\begin{equation} \label{CR generalized 1}
    \sum_{ j = 1}^{N}   {\lambda}_{j}  \ln    \frac{  { \left( 1 + \mathbf{n}_{{}_{ {\mathbf{V}}_{j} }} \right)}^{2}  }{  \sqrt{  \,- \mathcal{K} \, }  }  
\end{equation}
on the points where $\mathcal{K}<0$ and $1 + \mathbf{n}_{{}_{ {\mathbf{V}}_{j} }}>0$ for all $j \in \{1, \cdots, N \}$.
Then, the Chern-Ricci function is harmonic, and the conformally changed metric 
\begin{equation} \label{CR generalized 2}
   \left[ \prod_{ j = 1}^{N}      { \left( 1 + \mathbf{n}_{{}_{ {\mathbf{V}}_{j} }} \right)}^{2 {\lambda}_{j} }  \right] \mathbf{g}_{{}_{\Sigma}} 
\end{equation}
is flat. The case when $N=1$ indicates the flatness of the Chern metric (\ref{Chern flat}) in Proposition \ref{two flat}.
\end{lemma}
 
 \begin{proof}
By Proposition \ref{two flat}, each component function of the Chern-Ricci function (\ref{CR generalized 1}) is harmonic on $\Sigma$.  
 Recall that the Ricci metric $ {\mathbf{g}}_{{}_{\textrm{Ricci}}} = \sqrt{- {\mathcal{K}}_{ {\mathbf{g}}_{{}_{\Sigma}}} } \, {\mathbf{g}}_{{}_{\Sigma}}$ is flat.
 Since the Chern-Ricci function (\ref{CR generalized 1}) is harmonic,  by the curvature formula (\ref{curvature cc}) in Proposition \ref{two flat}, 
 the conformally changed metric  
 \[
  \left[ \prod_{ j = 1}^{N}      { \left( 1 + \mathbf{n}_{{}_{ {\mathbf{V}}_{j} }} \right)}^{2 {\lambda}_{j} }  \right] \mathbf{g}_{{}_{\Sigma}}  
  =       \prod_{ j = 1}^{N}  \, {\left[ \,  \frac{  { \left( 1 + \mathbf{n}_{{}_{ {\mathbf{V}}_{j} }} \right)}^{2 }  }{  \sqrt{  \,- \mathcal{K} \, }  } 
 \,  \right]}^{ {\lambda}_{j}}  \, {\mathbf{g}}_{{}_{\textrm{Ricci}}} 
 \] 
 should be also flat. 
 \end{proof}

\begin{lemma}[\textbf{Chern-Ricci harmonic functions in terms of the Weierstrass data}]   \label{CRW}
We identify a given constant unit vector field ${\mathbf{V}}$ by the point ${\pi}^{-1}\left( \alpha \right) \in {\mathbb{S}}^{2}$, 
where we denote $\pi : {\mathbb{S}}^{2} \to \overline{\, \mathbb{C} \,}$ the stereographic projection with respect to the north pole. 
 On a negatively curved minimal surface with the Weierstrass data $\left({\mathbf{G}}, d\mathbf{h}  \right)$ and the unit normal 
 vector field ${\mathbf{n}}={\pi}^{-1}\left(  \mathbf{G} \right)$, the Chern-Ricci harmonic functions
 associated to ${\mathbf{V}}$  (see Lemma (\ref{CR def})) are given in terms of the Weierstrass data as follows.
\begin{enumerate}
\item[\textbf{(a)}] 
\begin{equation}   \label{CR a}
  \ln \frac{ \, {\left( 1 - \mathbf{n}_{{}_{\mathbf{V}}} \right)}^{2} \, }{\sqrt{-\mathcal{K}}}   = \begin{cases}
 \ln \, \left[ \;  {\left( \frac{2}{\, 1 +{ \left\vert \alpha \right\vert}^{2} \,} \right)}^{2}  \left\vert \, \frac{  {\, ( \mathbf{G} - \alpha )}^{4}  \,}{  4\mathbf{G}} \, \frac{d\mathbf{h}}{d\mathbf{G}} \, \right\vert \; \right],  \quad \alpha \in \overline{\, \mathbb{C} \,} - 
  \{ \infty \}, \\
  \\
  \ln \, \left[ \;  {\left( \frac{2}{\, 1 + \frac{1}{{ \left\vert \alpha \right\vert}^{2}}  \,} \right)}^{2}  \left\vert \, \frac{  {\, ( \frac{1}{\,\alpha\,} {\mathbf{G}} - 1 )}^{4}  \,}{  4\mathbf{G}} \, \frac{d\mathbf{h}}{d\mathbf{G}} \, \right\vert \; \right],  \quad \alpha \in \overline{\, \mathbb{C} \,} - \{ 0 \}.
  \end{cases}
\end{equation} 
\item[\textbf{(b)}] 
\begin{equation}   \label{CR b}
  \ln \frac{ \, {\left( 1 + \mathbf{n}_{{}_{\mathbf{V}}} \right)}^{2} \, }{\sqrt{-\mathcal{K}}}   = \begin{cases}
 \ln \, \left[ \;  {\left( \frac{2}{\, 1 +{ \left\vert \alpha \right\vert}^{2} \,} \right)}^{2}  \left\vert \, \frac{  {\, ( \overline{\alpha} \mathbf{G} - 1 )}^{4}  \,}{  4\mathbf{G}} \, \frac{d\mathbf{h}}{d\mathbf{G}} \, \right\vert \; \right],  \quad \alpha \in \overline{\, \mathbb{C} \,} - 
  \{ \infty \}, \\
  \\
  \ln \, \left[ \;  {\left( \frac{2}{\, 1 + \frac{1}{{ \left\vert \alpha \right\vert}^{2}}  \,} \right)}^{2}  \left\vert \, \frac{  {\, ( {\mathbf{G}} - \frac{1}{\,\overline{\alpha}\,}  )}^{4}  \,}{  4\mathbf{G}} \, \frac{d\mathbf{h}}{d\mathbf{G}} \, \right\vert \; \right],  \quad \alpha \in \overline{\, \mathbb{C} \,} - \{ 0 \}.
  \end{cases}
\end{equation} 
\item[\textbf{(c)}]  
\begin{equation}    \label{CR c}
  \ln    \frac{ \, { \left( 1 - {\mathbf{n}_{{}_{\mathbf{V}}}}^{2}  \right)}  \,}{  \sqrt{  \,- \mathcal{K} \, }  }   = \begin{cases}
 \ln \, \left[ \;  {\left( \frac{2 \left\vert \alpha \right\vert }{\, 1 +{ \left\vert \alpha \right\vert}^{2} \,} \right)}^{2}  \left\vert \, \frac{  {( \mathbf{G} - \alpha )}^{2} 
  {( \mathbf{G} + \frac{1}{\,\overline{\alpha}\,} )}^{2} \,}{  4\mathbf{G}} \, \frac{d\mathbf{h}}{d\mathbf{G}} \, \right\vert \; \right],  \quad \alpha \in \overline{\, \mathbb{C} \,} - 
  \{0,  \infty \}, \\
  \\
  \ln \,    \left\vert \, \frac{  1 }{\,  \mathbf{G} \,} \, \frac{d\mathbf{h}}{d\mathbf{G}} \, \right\vert,  \quad \alpha \in  \{ 0, \infty \}.
  \end{cases}
\end{equation} 
\end{enumerate}
\end{lemma}
 
\begin{proof}  We first deduce the equalities in (\ref{CR a}). We need to compute the Jacobi field ${\mathbf{n}}_{{}_{\mathbf{V}}}= {\langle \mathbf{n}, {\mathbf{V}} \rangle}_{{\mathbb{R}}^{3}}$. By the definition of the stereographic projection $\pi : {\mathbb{S}}^{2} \to \overline{\, \mathbb{C} \,}$, we have
\[
  {\mathbf{n}}= \frac{1}{ {\left\vert \mathbf{G} \right\vert}^{2}  + 1 }
 \begin{bmatrix}
   \, 2 \, \textrm{Re} \left( \mathbf{G} \right)  \, \\
   \, 2 \, \textrm{Im} \left( \mathbf{G} \right)  \, \\
  \, {\left\vert \mathbf{G} \right\vert}^{2}  - 1 \,
 \end{bmatrix}
\quad 
\text{and} \quad 
  {\mathbf{V}}= \frac{1}{ {\left\vert \alpha \right\vert}^{2}  + 1 }
 \begin{bmatrix}
   \, 2 \, \textrm{Re} \left( \alpha \right)  \, \\
   \, 2 \, \textrm{Im} \left( \alpha \right)  \, \\
  \, {\left\vert \alpha \right\vert}^{2}  - 1 \,
 \end{bmatrix}.
\]
A straightforward computation yields 
\begin{equation} \label{Jacobi field}
    1 - \mathbf{n}_{{}_{\mathbf{V}}}   = \begin{cases}
  \frac{2}{\, 1 +{ \left\vert \alpha \right\vert}^{2} \,}  \frac{\,\left\vert  {( \mathbf{G} - \alpha )}^{2} \right\vert\, }{1 + {\vert \mathbf{G} \vert}^{2}  },  \quad \alpha \in \overline{\, \mathbb{C} \,} - 
  \{ \infty \}, \\
  \\
 \frac{2}{\, 1 + \frac{1}{{ \left\vert \alpha \right\vert}^{2}} \,}   \frac{\,\left\vert  {( \frac{1}{\alpha} {\mathbf{G}} - 1 )}^{2} \right\vert\, }{1 + {\vert \mathbf{G} \vert}^{2}  },  \quad \alpha \in \overline{\, \mathbb{C} \,} - \{ 0 \}.
  \end{cases}
\end{equation}
As well-known (for instance, see \cite[p. 5]{Kar2003} and \cite[Chapter 9]{Oss86}), the curvature ${\mathcal{K}}$ of the induced metric 
${\mathbf{g}}_{{}_{\Sigma}}  = \frac{1}{4} {\left( \frac{1}{  \, \vert \mathbf{G} \vert \,} + \vert \mathbf{G} \vert     \right)}^{2}  {\vert d \mathbf{h} \vert}^2$
is given in terms of the Weierstrass data: 
\begin{equation} \label{curvature K}
  \sqrt{-  {\mathcal{K}} \, }= \sqrt{ - {\mathcal{K}}_{ {\mathbf{g}}_{{}_{\Sigma}} } \,}= {\left(     \frac{2}{ \; \frac{1}{  \, \vert \mathbf{G} \vert \,} + \vert \mathbf{G} \vert \; }     \right)}^{2} 
   \left\vert \,  \frac{1}{\, \mathbf{G} \frac{\, d\mathbf{h} \,}{\, d \mathbf{G} \,} \, } \, \right\vert.
\end{equation}
Combining (\ref{Jacobi field}) and (\ref{curvature K}) gives the equalities in (\ref{CR a}). Replacing the pair $\left({\mathbf{V}}, \alpha \right)$ in (\ref{CR a}) by 
$\left(-{\mathbf{V}}, \frac{1}{\,\overline{\alpha}\,} \right)$ and observing the identity $\mathbf{n}_{{}_{ - \mathbf{V}}}= -  \mathbf{n}_{{}_{\mathbf{V}}} $ yield the equalities in (\ref{CR b}). Adding the equalities in (\ref{CR a}) and  (\ref{CR b}) gives the equalities in (\ref{CR c}).
\end{proof}

 \section{Classifications of minimal surfaces with constant Chern-Ricci functions}

We classify minimal surfaces in terms of Chern-Ricci harmonic functions, and examine the limit behaviors of periodic minimal surfaces. 

\begin{theorem}[\textbf{Rigidity of the associate family of Scherk's surfaces}] \label{Scherk Uniqueness Main}
Given a minimal surface with the Gauss curvature $\mathcal{K}<0$ and the unit normal $\mathbf{n}$, if there exists two \textbf{orthogonal} constant unit vector fields ${\mathbf{V}}_{1}$ and ${\mathbf{V}}_{2}$ such that the Chern-Ricci harmonic function (introduced in Lemma \ref{CR def}) 
\begin{equation} \label{Scherk two Jacobi Main}
\sum_{\mathbf{V} \in \left\{ {\mathbf{V}}_{1}, -{\mathbf{V}}_{1}, {\mathbf{V}}_{2}, - {\mathbf{V}}_{2} \right\}}  \ln \frac{\, {\left(1 -  {\mathbf{n}}_{{}_{{\mathbf{V}} }} 
   \right)}^2 \,}{\sqrt{ \, -{\mathcal{K}}} \,}  
  \end{equation}
is constant, then the minimal surface should be congruent to (a part of) a member of the associate family connecting a doubly periodic Scherk graph and a singly periodic Scherk tower (up to homotheties). 
\end{theorem}

\begin{proof} Without loss of generality, after applying rotations, we could take the identifications
{\small{
\[
{\mathbf{V}}_{1} = 
 \begin{bmatrix}
   \, 1  \, \\
   \, 0  \, \\
  \, 0 \,
 \end{bmatrix} = {\pi}^{-1} \left(1\right), \;\; 
-{\mathbf{V}}_{1} = 
 \begin{bmatrix}
   \, -1  \, \\
   \, 0  \, \\
  \, 0 \,
 \end{bmatrix}= {\pi}^{-1} \left(-1\right), \;\; 
 {\mathbf{V}}_{2} = 
 \begin{bmatrix}
   \, 0  \, \\
   \, 1  \, \\
  \, 0 \,
 \end{bmatrix}= {\pi}^{-1} \left(i\right), \;\; 
-{\mathbf{V}}_{2} = 
 \begin{bmatrix}
   \, 0  \, \\
   \, -1  \, \\
  \, 0 \,
 \end{bmatrix}= {\pi}^{-1} \left(-i\right).
\]
}}
Here, recall that $\pi : {\mathbb{S}}^{2} \to \overline{\, \mathbb{C} \,}$ denotes the stereographic projection from the north pole. 
Let  $\left({\mathbf{G}}, d\mathbf{h}  \right)$ denote the Weierstrass data. We use the formula (\ref{CR a}) in Lemma \ref{CRW} 
{\small{
 \[
 \ln \frac{ \, {\left( 1 - \mathbf{n}_{{}_{\mathbf{V}}} \right)}^{2} \, }{\sqrt{-\mathcal{K}}}   = 
 \ln \, \left[ \;  {\left( \frac{2}{\, 1 +{ \left\vert \alpha \right\vert}^{2} \,} \right)}^{2}  \left\vert \, \frac{  {\, ( \mathbf{G} - \alpha )}^{4}  \,}{  4\mathbf{G}} \, \frac{d\mathbf{h}}{\,d\mathbf{G}\,} \, \right\vert \; \right] 
\]
}}
for $\alpha \in \{1, -1, i, -1 \}$. The constancy of the function (\ref{Scherk two Jacobi Main}) guarantees that the product function
{\small{
\[
  \left\vert    \,   \frac{1}{  {\mathbf{G}}^{4} }  {\left( \frac{d\mathbf{h}}{\,d\mathbf{G}\,} \right)}^{4} \,
   \prod_{ \alpha \in \{1, -1, i, -1 \}  }    ( \mathbf{G} - \alpha )^4   \right\vert \, = 
     {\left\vert    \,   \frac{1}{  {\mathbf{G}}  }   \frac{d\mathbf{h}}{\,d\mathbf{G}\,}    \,
     ( {\mathbf{G}}^4 - 1 )   \right\vert}^{4} \,
\]
}}
is a positive constant. Hence, there exist constants $\rho>0$ and $\Theta \in \left[0, 2\pi \right]$ such that 
\[
\frac{1}{  {\mathbf{G}}  }   {d\mathbf{h}} = \frac{\rho \, {e}^{i\Theta} }{{\mathbf{G}}^4 - 1} {d\mathbf{G}}.
\] 
\end{proof}

\begin{remark}[\textbf{Rigidity of Scherk's surfaces in terms of two flat structures}] \label{StwoFS} We first explain that Theorem \ref{Scherk Uniqueness Main} is equivalent to Theorem \ref{Scherk Uniqueness}.
Indeed, taking account into the identity $\mathbf{n}_{{}_{ - \mathbf{V}}}= -  \mathbf{n}_{{}_{\mathbf{V}}}$ and letting $\mathcal{V}:=\left\{ {\mathbf{V}}_{1}, -{\mathbf{V}}_{1}, {\mathbf{V}}_{2}, - {\mathbf{V}}_{2} \right\}$, we have
\[
 \sum_{\mathbf{V} \in \mathcal{V}}  \ln \frac{\, {\left(1 -  {\mathbf{n}}_{{}_{{\mathbf{V}} }} 
   \right)}^2 \,}{\sqrt{ \, -{\mathcal{K}}} \,} =  \sum_{\mathbf{V} \in \mathcal{V}}  \ln \frac{\, {\left(1 +  {\mathbf{n}}_{{}_{{\mathbf{V}} }} 
   \right)}^2 \,}{\sqrt{ \, -{\mathcal{K}}} \,} = 2\sum_{\mathbf{V} \in  \left\{ {\mathbf{V}}_{1},   {\mathbf{V}}_{2}  \right\} }  \ln \frac{\,  1 -  { {\mathbf{n}}_{{}_{{\mathbf{V}} }} }^{2}  \,}{\sqrt{ \, -{\mathcal{K}}} \,} =4   \ln \frac{\, \sqrt{ \left( 1 -  { {\mathbf{n}}_{{}_{ {{\mathbf{V}}_{1} } }} }^{2} \right)  \left( 1 -  { {\mathbf{n}}_{{}_{ {{\mathbf{V}}_{2} } }} }^{2} \right)} \,}{\sqrt{ \, -{\mathcal{K}}} \,}. 
\] 
Theorem \ref{Scherk Uniqueness Main} shows that, on a minimal surface, the flat metric
\[
\sqrt{ \left( 1 -  { {\mathbf{n}}_{{}_{ {{\mathbf{V}}_{1} } }} }^{2} \right)\left( 1 -  { {\mathbf{n}}_{{}_{ {{\mathbf{V}}_{2} } }} }^{2} \right)    } \,  {\mathbf{g}}_{{}_{\Sigma}} 
\]
could become a constant multiple of the flat metric ${\mathbf{g}}_{{}_{\textrm{Ricci}}} = \sqrt{- {\mathcal{K}}_{ {\mathbf{g}}_{{}_{\Sigma}}} } \, {\mathbf{g}}_{{}_{\Sigma}}$,
only when it belongs to the associate family of Scherk's surfaces.
\end{remark}

 \begin{figure}[H]
 \centering
 \includegraphics[height=4.250cm]{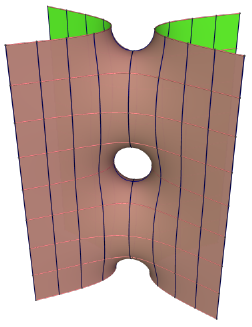}
 \caption{\small{An approximation \cite{WebEnn} of a part of a \textbf{sheared} singly periodic Scherk's surface}} 
 \end{figure}

\begin{theorem}[\textbf{Rigidity of the associate family of generalized Scherk's surfaces}] \label{Scherk Uniqueness Main 2}
Given a minimal surface with the Gauss curvature $\mathcal{K}<0$ and the unit normal $\mathbf{n}$, if there exists two distinct unit vector fields ${\mathbf{V}}_{1}$ and ${\mathbf{V}}_{2}$ such that the Chern-Ricci harmonic function 
\begin{equation} \label{Scherk two Jacobi Main 2}
\sum_{\mathbf{V} \in \left\{ {\mathbf{V}}_{1}, -{\mathbf{V}}_{1}, {\mathbf{V}}_{2}, - {\mathbf{V}}_{2} \right\}}  \ln \frac{\, {\left(1 -  {\mathbf{n}}_{{}_{{\mathbf{V}} }} 
   \right)}^2 \,}{\sqrt{ \, -{\mathcal{K}}} \,}  
  \end{equation}
is constant, then the minimal surface should be congruent to (a part of) a member of the associate family connecting a generalized doubly periodic Scherk graph and a generalized singly periodic Scherk tower.  
\end{theorem}

\begin{proof} Without loss of generality, after applying rotations, we could take the identifications
{\small{
\[
{\mathbf{V}}_{1} =   {\pi}^{-1} \left( e^{i \theta} \right), \;\; 
-{\mathbf{V}}_{1} =  {\pi}^{-1} \left( - e^{i \theta} \right), \;\; 
{\mathbf{V}}_{2} =   {\pi}^{-1} \left( e^{-i \theta} \right), \;\; 
-{\mathbf{V}}_{2} =  {\pi}^{-1} \left( - e^{-i \theta} \right),
\]
}}
where $\pi : {\mathbb{S}}^{2} \to \overline{\, \mathbb{C} \,}$ is the stereographic projection. We use the formula (\ref{CR a}) in Lemma \ref{CRW} 
 \[
 \ln \frac{ \, {\left( 1 - \mathbf{n}_{{}_{\mathbf{V}}} \right)}^{2} \, }{\sqrt{-\mathcal{K}}}   = 
 \ln \, \left[ \;  {\left( \frac{2}{\, 1 +{ \left\vert \alpha \right\vert}^{2} \,} \right)}^{2}  \left\vert \, \frac{  {\, ( \mathbf{G} - \alpha )}^{4}  \,}{  4\mathbf{G}} \, \frac{d\mathbf{h}}{\,d\mathbf{G}\,} \, \right\vert \; \right]
\]
for $\alpha   \in \{e^{i \theta}, -e^{i \theta}, e^{-i \theta}, -e^{-i \theta} \}$.
The constancy of the function (\ref{Scherk two Jacobi Main 2}) guarantees that the product 
\[
  \left\vert    \,   \frac{1}{  {\mathbf{G}}^{4} }  {\left( \frac{d\mathbf{h}}{\,d\mathbf{G}\,} \right)}^{4} \,
   \prod_{ \alpha \in \{e^{i \theta}, -e^{i \theta}, e^{-i \theta}, -e^{-i \theta} \}  }    ( \mathbf{G} - \alpha )^4   \right\vert \, = 
     {\left\vert    \,   \frac{1}{  {\mathbf{G}}  }   \frac{d\mathbf{h}}{\,d\mathbf{G}\,}    \,
    \left( {\mathbf{G}}^{2} - {e^{2i \theta}} \right) \left( {\mathbf{G}}^{2} - {e^{-2i \theta}} \right)    \right\vert}^{4} \,
\]
is a positive constant. We recover the Weierstrass data for the associate family of generalized Scherk's surfaces \cite{Douglas2014, Kar1988, Weber2001}:
\[
\frac{1}{  {\mathbf{G}}  }  {d\mathbf{h}}  = \frac{\rho \, {e}^{i\Theta} }{\,  \left( {\mathbf{G}}^{2} - {e^{2i \theta}} \right) \left( {\mathbf{G}}^{2} - {e^{-2i \theta}} \right)  \, } {d\mathbf{G}}
\] 
for some constants $\rho>0$ and $\Theta \in \left[0, 2\pi \right]$.    
\end{proof}

\begin{example}[\textbf{triply periodic tCLP surfaces} connecting \textbf{doubly periodic Scherk surfaces} and \textbf{singly periodic Scherk surfaces}] \label{CLP data}
Imagine eight points on the unit circle ${\mathbb{S}}^{1}$ in the $xy$-plane:
\begin{equation} \label{tCLP vertices}
\pm {\mathbf{V}}_{1} = 
\pm \begin{bmatrix}
   \, \cos \theta  \, \\
   \, \sin \theta  \, \\
  \, 0  \,
 \end{bmatrix}, \;\; 
\pm{\mathbf{V}}_{2} = 
\pm \begin{bmatrix}
    \, \sin \theta  \, \\
   \, \cos \theta  \, \\
  \, 0 \,
 \end{bmatrix}, \;\; 
\pm {\mathbf{V}}_{3} = 
\pm \begin{bmatrix}
    \, - \sin \theta  \, \\
    \, \cos \theta  \, \\
  \, 0  \,
 \end{bmatrix}, \;\; 
\pm {\mathbf{V}}_{4} = 
 \pm \begin{bmatrix}
    \, - \cos \theta \, \\
    \, \sin \theta  \, \\
   \, 0 \,
 \end{bmatrix}.
\end{equation}
For $\theta \in \left(0, \frac{\pi}{4} \right)$,  after identifying these eight vertices in ${\mathbb{S}}^{1} \subset {\mathbb{S}}^{2} \subset {\mathbb{R}}^{3}$ as the constant unit vector fields, we want to construct minimal surfaces so that the Chern-Ricci harmonic function
\begin{equation} \label{D one}
\sum_{\mathbf{V} \in \left\{ \pm {\mathbf{V}}_{1}, \pm {\mathbf{V}}_{2}, \pm {\mathbf{V}}_{3}, \pm {\mathbf{V}}_{4} \right\}}  \ln \frac{\, {\left(1 -  {\mathbf{n}}_{{}_{{\mathbf{V}} }}   \right)}^2 \,}{\sqrt{ \, -{\mathcal{K}}} \,}   
 \end{equation}
is constant.  When $\theta \shortarrow{7} 0$ or $\theta \shortarrow{1} \frac{\pi}{4}$, can we 
geometrically prescribe the limit minimal surfaces?  We observe that the eight vertices consist of two congruent rectangles in the $xy$-plane. 
\begin{enumerate}
\item[\textbf{(a)}] When $\theta \shortarrow{7} 0$, two congruent rectangles \emph{collapses} to the union of two orthogonal segments. 
We expect that the limit surfaces has the constant Chern-Ricci harmonic function
\begin{equation} \label{D one 1}
\sum_{\mathbf{V} \in \left\{ \pm {\left(1, 0, 0 \right)}^{\intercal}, \, \pm {\left(0, 1, 0 \right)}^{\intercal} \right\}}  \ln \frac{\, {\left(1 -  {\mathbf{n}}_{{}_{{\mathbf{V}} }}   \right)}^2 \,}{\sqrt{ \, -{\mathcal{K}}} \,}.   
 \end{equation}
\item[\textbf{(b)}] When $\theta \shortarrow{1} \frac{\pi}{4}$, two congruent rectangles \emph{collapses} to the union of the orthogonal segments. We expect that the limit surfaces has the constant Chern-Ricci harmonic function
\begin{equation} \label{D one 2}
\sum_{\mathbf{V} \in \left\{ \pm {\left(\cos \left( \frac{\pi}{4} \right),   \sin \left( \frac{\pi}{4} \right), 0 \right)}^{\intercal}, \, \pm {\left(\cos \left( \frac{3\pi}{4} \right), 
\sin \left( \frac{3\pi}{4} \right), 0 \right)}^{\intercal} \right\}}  \ln \frac{\, {\left(1 -  {\mathbf{n}}_{{}_{{\mathbf{V}} }}   \right)}^2 \,}{\sqrt{ \, -{\mathcal{K}}} \,}.   
 \end{equation}
\item[] 
\end{enumerate}
Theorem \ref{Scherk Uniqueness Main} guarantees that such limit surfaces are members of the associate family from a \textbf{doubly periodic Scherk surface} to a \textbf{singly periodic Scherk surface}. We find minimal surfaces with a constant Chern-Ricci harmonic function in (\ref{D one}).
For  $\theta \in \left(0, \frac{\pi}{4} \right)$, we have the image  
\[
  \mathcal{A} := {\pi} \left(\, \left\{ \pm {\mathbf{V}}_{1}, \pm {\mathbf{V}}_{2}, \pm {\mathbf{V}}_{3}, \pm {\mathbf{V}}_{4} \right\} \, \right) 
   = \left\{ \, \pm e^{i \theta}, \,   \pm i e^{i \theta}, \,   \pm e^{-i \theta}, \,  \pm i e^{-i \theta} \, \right\},
\]
where $\pi : {\mathbb{S}}^{2} \to \overline{\, \mathbb{C} \,}$ is the stereographic projection. According to the formula (\ref{CR a}) in Lemma \ref{CRW} 
{\small{
 \[
 \ln \frac{ \, {\left( 1 - \mathbf{n}_{{}_{\mathbf{V}}} \right)}^{2} \, }{\sqrt{-\mathcal{K}}}   = 
 \ln \, \left[ \;  {\left( \frac{2}{\, 1 +{ \left\vert \alpha \right\vert}^{2} \,} \right)}^{2}  \left\vert \, \frac{  {\, ( \mathbf{G} - \alpha )}^{4}  \,}{  4\mathbf{G}} \, \frac{d\mathbf{h}}{\,d\mathbf{G}\,} \, \right\vert \; \right], 
\]
}}
we see that the constancy of the prescribed Chern-Ricci function (\ref{D one}) implies that constancy of   
\[
  \left\vert    \,   \frac{1}{  {\mathbf{G}}^{8} }  {\left( \frac{d\mathbf{h}}{\,d\mathbf{G}\,} \right)}^{8} \,
       {\left( {\mathbf{G}}^{2} - {e}^{2i \theta} \right)}^4   {\left( {\mathbf{G}}^{2} + {e}^{2i \theta} \right)}^4   {\left( {\mathbf{G}}^{2} - {e}^{-2i \theta} \right)}^4   {\left( {\mathbf{G}}^{2} + {e}^{-2i \theta} \right)}^4      \, \right\vert.  
\]
This recovers, up to associate families, the {triply periodic} minimal surface ${\Sigma}_{\lambda \in \left( -2, 2\right)}$ in \textbf{tCLP family} \cite{ES2018, KPS2014} with the Weierstrass data  
\[
 \frac{1}{ {\mathbf{G}} }  {d\mathbf{h}}  = \frac{1}{\, \sqrt{ {\mathbf{G}}^{8} + \lambda {\mathbf{G}}^{4} + 1 }\,} d{\mathbf{G} \, }, \quad
 \lambda = - 2 \cos \left( 4 \theta \right) \in \left(-2, 2\right).
\]
\begin{enumerate}
\item[\textbf{(a)}] The limit surface  ${\Sigma}_{-2}$ recovers the \textbf{doubly periodic Scherk surface} with   
\[
 \frac{1}{ {\mathbf{G}} }  {d\mathbf{h}}  = \frac{1}{\, {\mathbf{G}}^{4} - 1 \,} d{\mathbf{G} \, }. 
\]
\item[\textbf{(b)}] The limit surface  ${\Sigma}_{2}$ recovers the \textbf{singly periodic Scherk surface} with   
\[
 \frac{1}{ {\mathbf{G}} }  {d\mathbf{h}}  = \frac{1}{\, {\mathbf{G}}^{4} + 1 \,} d{\mathbf{G} \, }. 
\]
\end{enumerate}
\end{example}

\begin{example}[From \textbf{triply periodic Schwarz D surface} to \textbf{doubly periodic Scherk surface}] \label{tD data} Imagine a cube 
inscribed in the round unit 
sphere ${\mathbb{S}}^{2}$ sitting in ${\mathbb{R}}^{3}$. The side length of the cube is $\sqrt{\frac{4}{3}}$. Taking $\theta={\sin}^{-1}\left( \, \sqrt{{\frac{1}{3}}}\, \right)$, we obtain an example of such cube with the vertices in ${\mathbb{S}}^{2}$:
{\small{
\begin{equation} \label{tD points}
\pm {\mathbf{V}}_{1} = 
\pm \begin{bmatrix}
   \, \cos \theta  \, \\
   \, 0  \, \\
  \, \sin \theta \,
 \end{bmatrix}, \;\; 
\pm{\mathbf{V}}_{2} = 
\pm \begin{bmatrix}
    \, \cos \theta  \, \\
   \, 0  \, \\
  \, - \sin \theta \,
 \end{bmatrix}, \;\; 
\pm {\mathbf{V}}_{3} = 
\pm \begin{bmatrix}
    \, 0  \, \\
    \, \cos \theta  \, \\
  \, \sin \theta \,
 \end{bmatrix}, \;\; 
\pm {\mathbf{V}}_{4} = 
 \pm \begin{bmatrix}
    \, 0  \, \\
    \, \cos \theta  \, \\
   \, -\sin \theta \,
 \end{bmatrix}.
\end{equation}
}}
For  $\theta \in \left(0, \frac{\pi}{2} \right)$, after identifying these eight vertices in ${\mathbb{S}}^{2}$ as the constant unit vector fields, we want to construct minimal surfaces so that the Chern-Ricci harmonic function
\begin{equation} \label{tD CR}
\sum_{\mathbf{V} \in \left\{ \pm {\mathbf{V}}_{1}, \pm {\mathbf{V}}_{2}, \pm {\mathbf{V}}_{3}, \pm {\mathbf{V}}_{4} \right\}}  \ln \frac{\, {\left(1 -  {\mathbf{n}}_{{}_{{\mathbf{V}} }}   \right)}^2 \,}{\sqrt{ \, -{\mathcal{K}}} \,}   
 \end{equation}
is constant. For  $\theta \in \left(0, \frac{\pi}{2} \right)$, taking $a=\tan \left( \frac{\pi}{4} - \frac{\theta}{2} \right)$, we have the image  
\[
  \mathcal{A} := {\pi} \left(\, \left\{ \pm {\mathbf{V}}_{1}, \pm {\mathbf{V}}_{2}, \pm {\mathbf{V}}_{3}, \pm {\mathbf{V}}_{4} \right\} \, \right) 
   = \left\{a, - \frac{1}{a}, a i , - \frac{1}{ \, \overline{a i}} = \frac{i}{a}, -a, \frac{1}{a}, -a i, - \frac{1}{ \, \overline{-a i}} = -\frac{i}{a} \, \right\},
\]
where $\pi : {\mathbb{S}}^{2} \to \overline{\, \mathbb{C} \,}$ is the stereographic projection. According to the formula (\ref{CR a}) in Lemma \ref{CRW} 
{\small{
 \[
 \ln \frac{ \, {\left( 1 - \mathbf{n}_{{}_{\mathbf{V}}} \right)}^{2} \, }{\sqrt{-\mathcal{K}}}   = 
 \ln \, \left[ \;  {\left( \frac{2}{\, 1 +{ \left\vert \alpha \right\vert}^{2} \,} \right)}^{2}  \left\vert \, \frac{  {\, ( \mathbf{G} - \alpha )}^{4}  \,}{  4\mathbf{G}} \, \frac{d\mathbf{h}}{\,d\mathbf{G}\,} \, \right\vert \; \right], 
\]
}}
we see that the constancy of the prescribed Chern-Ricci function (\ref{D one}) implies that constancy of   
\[
  \left\vert    \,   \frac{1}{  {\mathbf{G}}^{8} }  {\left( \frac{d\mathbf{h}}{\,d\mathbf{G}\,} \right)}^{8} \,
       {\left( {\mathbf{G}}^{4} - {a}^{4} \right)}^4  {\left( {\mathbf{G}}^{4} - \frac{1}{{a}^{4}} \right)}^4     \, \right\vert.  
\]
This recovers the {triply periodic} minimal surface ${\Sigma}_{\lambda \in \left( -\infty, -2\right)}$ in \textbf{tD family} determined by 
\[
 \frac{1}{ {\mathbf{G}} }  {d\mathbf{h}}  = \frac{1}{\, \sqrt{ {\mathbf{G}}^{8} + \lambda {\mathbf{G}}^{4} + 1 }\,} d{\mathbf{G} \, }, 
\]
up to associate families (and homotheties). It is straightforward to check that   
\[
  \lambda = - \left(  {a}^{4} + \frac{1}{ {a}^{4} } \right) = 2 -   {\left[ \,  a^{2}+ \frac{1}{a^2} \, \right]}^{2} = - 2 - 16 \frac{{\sin}^{2} \theta}{{\cos}^{4} \theta}<-2.
\]
\begin{enumerate}
\item[\textbf{(a)}] For $\theta={\sin}^{-1}\left( \, \sqrt{{\frac{1}{3}}}\, \right)$, we find that the eight vertices in (\ref{tD points}) consist of a cube inscribed in the unit 
sphere, and meet \textbf{Schwarz diamond surface} ${\Sigma}_{-14}$. 
\item[\textbf{(b)}] The limit surface ${\Sigma}_{-2}$ is the doubly periodic Scherk surface with the Weierstrass data
\[
    \frac{1}{ {\mathbf{G}} }  {d\mathbf{h}}    =   \frac{1}{\,    {\mathbf{G}}^{4}  - 1    \,}  d{\mathbf{G} \, }.
\]
\end{enumerate}
\end{example}

  \begin{figure}[H]
 \centering
 \includegraphics[height=4.20cm]{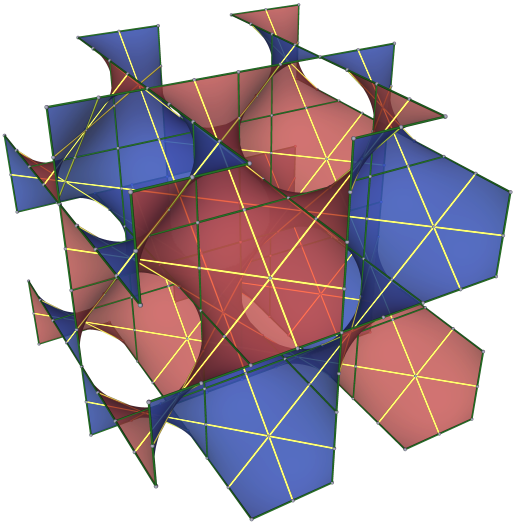} \quad \quad 
  \includegraphics[height=4.20cm]{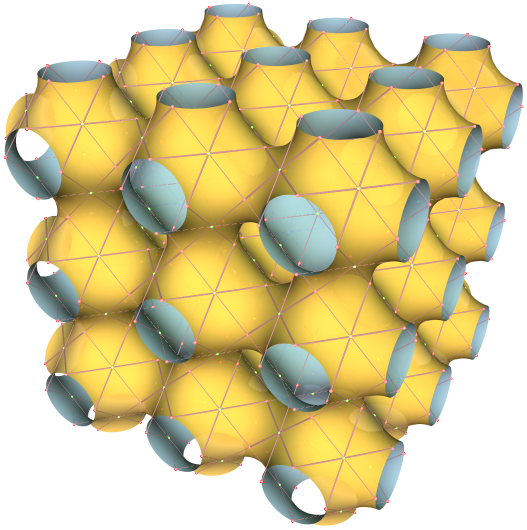}
 \caption{\small{Approximations \cite{WebEnn} of parts of triply periodic Schwarz D and P surfaces}} 
 \end{figure}

\begin{example}[From \textbf{triply periodic Schwarz P surface} to \textbf{singly periodic Scherk tower}] \label{tP data}
The {triply periodic} minimal surface ${\Sigma}_{\lambda \in \left( 2, \infty \right)}$ in \textbf{tP family} admits the Weierstrass data 
\[
 \frac{1}{ {\mathbf{G}} }  {d\mathbf{h}}  = \frac{1}{\, \sqrt{ {\mathbf{G}}^{8} + \lambda {\mathbf{G}}^{4} + 1 }\,} d{\mathbf{G} \, }. 
\]
\begin{enumerate}
\item[\textbf{(a)}] ${\Sigma}_{14}$ becomes the \textbf{Schwarz primitive surface}. 
\item[\textbf{(b)}] The limit surface ${\Sigma}_{2}$ induces the  \textbf{singly periodic Scherk tower} with the Weierstrass data
\[
    \frac{1}{ {\mathbf{G}} }  {d\mathbf{h}}   =   \frac{1}{\,    {\mathbf{G}}^{4}  + 1    \,}  d{\mathbf{G} \, }.
\]
\end{enumerate}
\end{example}

\begin{remark}
Meeks' family \cite[p. 912]{Meeks 1990} includes minimal surfaces in Example \ref{CLP data}, \ref{tD data}, \ref{tP data}.
\end{remark}

\begin{proposition}[{Conjugate surfaces of minimal surfaces in Example \ref{CLP data}, \ref{tD data}, \ref{tP data}}] 
Let ${\Sigma}_{\lambda \in \mathbb{R}}$ denote the minimal surface with the Weierstrass data
 \[
   {d\mathbf{h}}  = \frac{ {\mathbf{G}}}{\, \sqrt{ {\mathbf{G}}^{8} + \lambda {\mathbf{G}}^{4} + 1 }\,} d{\mathbf{G} \, }. 
\]
Then, the conjugate minimal surface of ${\Sigma}_{\lambda}$ is congruent to the minimal surface ${\Sigma}_{-\lambda}$. See also \cite[Example 6.3]{KPS2014}. In particular, the 
Schwarz CLP surface ${\Sigma}_{0}$ \cite{EFS2015, Kar1989, Schwarz1890} is self-conjugate. The Schwarz P surface ${\Sigma}_{14}$ in the tP family is conjugate to 
the Schwarz D surface ${\Sigma}_{-14}$ in the tD family.
\end{proposition} 

\begin{proof}
Since $\left(  e^{i \frac{\pi}{4}}  \right)^{2}=i$, $\left(  e^{i \frac{\pi}{4}}  \right)^{4}=-1$, and $\left(  e^{i \frac{\pi}{4}}  \right)^{8}=1$, the rotated Gauss map 
$\widetilde{\,\mathbf{G}\,} :=  e^{i \frac{\pi}{4}}  \mathbf{G}$ solves 
\[
\frac{i \, {\mathbf{G}}}{\, \sqrt{ {\mathbf{G}}^{8} + \lambda {\mathbf{G}}^{4} + 1 }\,} d{\mathbf{G} \, }= \frac{ \widetilde{\,\mathbf{G}\,}}{\, \sqrt{ {  {\, \widetilde{\,\mathbf{G}\,} \,}   }^{8} - \lambda {{\, \widetilde{\,\mathbf{G}\,} \,}  }^{4} + 1 }\,} d{{\widetilde{\,\mathbf{G}\,} \,}  }. 
\]
\end{proof}

   \begin{figure}[H]
 \centering
 \includegraphics[height=4.05cm]{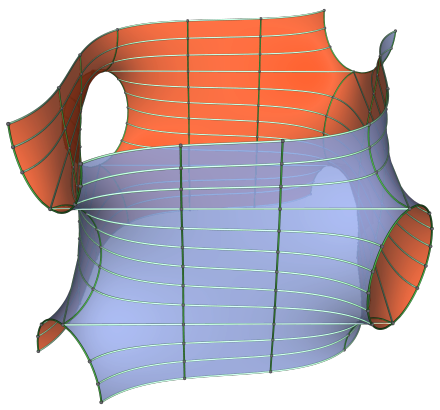} 
  \quad \quad  \includegraphics[height=4.20cm]{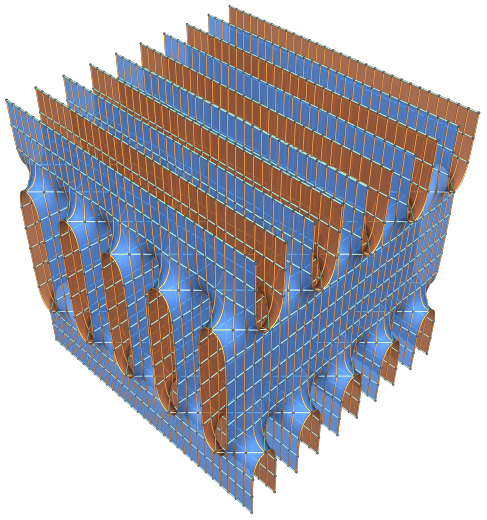}
 \caption{\small{Approximations \cite{WebEnn} of parts of triply periodic Schwarz CLP surfaces}} 
 \end{figure}

\begin{example}[{\textbf{rPD family} containing both \textbf{Schwarz P surface} and  \textbf{Schwarz D surface}}] \label{rPD data} 
 Imagine a regular tetrahedron inscribed in the round unit 
sphere ${\mathbb{S}}^{2}$ sitting in ${\mathbb{R}}^{3}$. The side length of the regular tetrahedron is $\sqrt{\frac{8}{3}}$. Taking $\theta={\sin}^{-1}\left( \, {\frac{1}{3}}\, \right)$, we obtain an example of such regular tetrahedron with the following vertices in ${\mathbb{S}}^{2}$:
\begin{equation} \label{rPD points}
  {\mathbf{V}}_{0} = 
  \begin{bmatrix}
   \, 0  \, \\
   \, 0  \, \\
  \, 1 \,
 \end{bmatrix}, \;\; 
 {\mathbf{V}}_{1} = 
  \begin{bmatrix}
    \, \cos \theta  \, \\
   \, 0  \, \\
  \, - \sin \theta \,
 \end{bmatrix}, \;\; 
  {\mathbf{V}}_{2} = 
  \begin{bmatrix}
    \, \cos \theta \cos  \left( \frac{2 \pi}{3} \right) \, \\
    \, \cos \theta \sin  \left( \frac{2 \pi}{3} \right) \, \\
  \, - \sin \theta \,
 \end{bmatrix}, \;\; 
  {\mathbf{V}}_{3} = 
   \begin{bmatrix}
   \, \cos \theta \cos  \left( - \frac{2 \pi}{3} \right) \, \\
    \, \cos \theta \sin  \left( - \frac{2 \pi}{3} \right) \, \\
  \, - \sin \theta \,
 \end{bmatrix}.
\end{equation}
For  $\theta \in \left( - \frac{\pi}{2}, \frac{\pi}{2} \right)$, after identifying the 8 points $\pm {\mathbf{V}}_{0}, \pm {\mathbf{V}}_{1}, \pm {\mathbf{V}}_{2}, \pm {\mathbf{V}}_{3}$ in ${\mathbb{S}}^{2}$ as the constant unit vector fields, we construct minimal surfaces such that the Chern-Ricci harmonic function
\begin{equation} \label{rPD one}
\sum_{\mathbf{V} \in \left\{ \pm {\mathbf{V}}_{0}, \pm {\mathbf{V}}_{1}, \pm {\mathbf{V}}_{2}, \pm {\mathbf{V}}_{3} \right\}}  \ln \frac{\, {\left(1 -  {\mathbf{n}}_{{}_{{\mathbf{V}} }}   \right)}^2 \,}{\sqrt{ \, -{\mathcal{K}}} \,}   
 \end{equation}
is constant. Taking $a=\tan \left( \frac{\pi}{4} - \frac{\theta}{2} \right)$ and $\omega=e^{i \frac{2 \pi }{3}}$, we have the image  
\[
  \mathcal{A} := {\pi} \left(\, \left\{ \pm {\mathbf{V}}_{0}, \pm {\mathbf{V}}_{1}, \pm {\mathbf{V}}_{2}, \pm {\mathbf{V}}_{3} \right\} \, \right) 
   = \left\{\infty,  a,   a {\omega}, a{ {\omega}^{2} }, 0, - \frac{1}{a},   - \frac{1}{a} {\omega},  - \frac{1}{a} {\omega}^{2}  \, \right\},
\]
where $\pi : {\mathbb{S}}^{2} \to \overline{\, \mathbb{C} \,}$ denotes the stereographic projection from the north pole.
To determine the Weierstrass data of desired minimal surfaces, we shall use the formula (\ref{CR a}) in Lemma \ref{CRW} 
{\small{
\[
  \ln \frac{ \, {\left( 1 - \mathbf{n}_{{}_{\mathbf{V}}} \right)}^{2} \, }{\sqrt{-\mathcal{K}}}   = \begin{cases}
   \ln \, \left\vert \;  {\mathbf{G}}^3 \frac{d\mathbf{h}}{d\mathbf{G}} \; \right\vert,  \quad \alpha =0, \\
 \ln \, \left[ \;  {\left( \frac{2}{\, 1 +{ \left\vert \alpha \right\vert}^{2} \,} \right)}^{2}  \left\vert \, \frac{  {\, ( \mathbf{G} - \alpha )}^{4}  \,}{  4\mathbf{G}} \, \frac{d\mathbf{h}}{d\mathbf{G}} \, \right\vert \; \right],  \quad \alpha \in \overline{\, \mathbb{C} \,} - 
  \{ 0, \infty \}, \\
  \ln \, \left\vert \; \frac{1}{\mathbf{G}} \frac{d\mathbf{h}}{d\mathbf{G}} \; \right\vert,  \quad \alpha =\infty.
  \end{cases}
\]
}}
We see that the constancy of the prescribed Chern-Ricci function (\ref{rPD one}) implies that constancy of   
\[
  \left\vert    \,   \frac{1}{  {\mathbf{G}}^{4} }  \,  {\left( {\mathbf{G}}^{3} - {a}^{3} \right)}^4   {\left( {\mathbf{G}}^{3} + \frac{1}{a^3} \right)}^4  
   \,  {\left( \frac{d\mathbf{h}}{\,d\mathbf{G}\,} \right)}^{8}    \, \right\vert.  
\]
This recovers, up to associate families, the {triply periodic} minimal surface ${\Sigma}_{a>0}$ in the \textbf{rPD family} \cite{Fogden1993, FH1999, Schwarz1890} with 
the Weierstrass data  
\[
 \frac{1}{ {\mathbf{G}} }  {d\mathbf{h}}  = \frac{1}{\, \sqrt{  {\mathbf{G}} {\left( {\mathbf{G}}^{3} - {a}^{3} \right)}    {\left( {\mathbf{G}}^{3} + \frac{1}{a^3} \right)}  }\,} d{\mathbf{G} \, }, \quad a>0.
\]
\begin{enumerate} 
\item[\textbf{(a)}] The conjugate surface of the minimal surface ${\Sigma}_{a}$ is congruent to ${\Sigma}_{\frac{1}{a}}$. In particular, the 
 minimal surface ${\Sigma}_{1}$ is self-conjugate.  
\item[\textbf{(b)}] When $a= \frac{1}{\sqrt{2}}$, we have $\theta={\sin}^{-1}\left( \, {\frac{1}{3}}\, \right)$. The eight vertices 
\[
\pm {\mathbf{V}}_{0}, \pm {\mathbf{V}}_{1}, \pm {\mathbf{V}}_{2}, \pm {\mathbf{V}}_{3}
\]
consists of a cube inscribed in ${\mathbb{S}}^{2}$. The surface ${\Sigma}_{\frac{1}{\sqrt{2}}}$ is the \textbf{Schwarz P surface}.  
\item[\textbf{(c)}] When $a = \sqrt{2}$, we have $\theta={\sin}^{-1}\left( \, - {\frac{1}{3}}\, \right)$. The eight vertices 
\[
\pm {\mathbf{V}}_{0}, \pm {\mathbf{V}}_{1}, \pm {\mathbf{V}}_{2}, \pm {\mathbf{V}}_{3}
\]
consists of a cube inscribed in ${\mathbb{S}}^{2}$.  The surface ${\Sigma}_{\sqrt{2}}$ is the \textbf{Schwarz D surface}.  
\end{enumerate}
 \end{example}

 \begin{figure}[H]
 \centering
 \includegraphics[height=4.20cm]{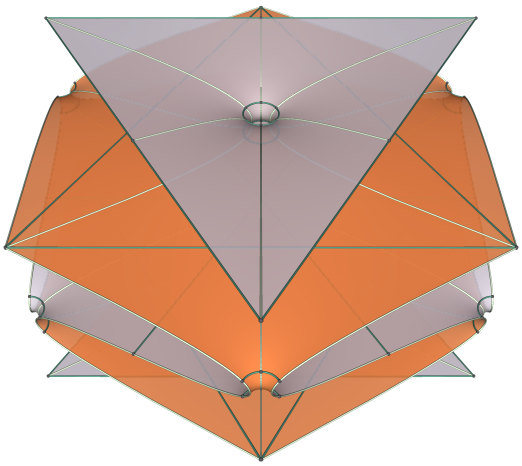}   \quad \quad   \includegraphics[height=4.20cm]{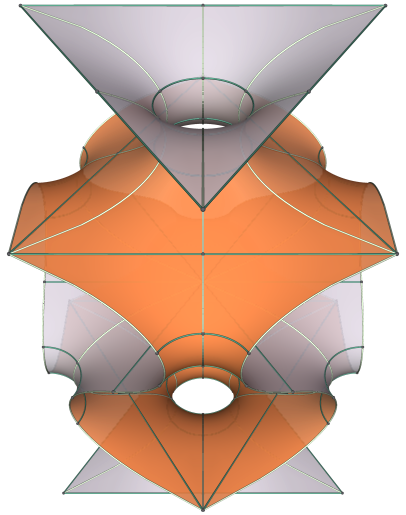} 
 \caption{\small{Approximations \cite{WebEnn} of parts of triply periodic Schwarz rPD surfaces}} 
 \end{figure}

\begin{example}[{\textbf{Karcher's saddle tower with 6 ends} as the limit of surfaces in \textbf{hCLP surfaces}}] \label{hCLP data} 
Fix the angle $\theta \in \left(0, \frac{\pi}{3}\right)$. Identifying the following eight points in ${\mathbb{S}}^{2}$   
\begin{equation} \label{hCLP points 1}
 {\mathbf{V}}_{1} = 
  \begin{bmatrix}
    \, \cos \theta  \, \\
   \,  \sin \theta  \, \\
  \, 0 \,
 \end{bmatrix}, \;\; 
  {\mathbf{V}}_{2} = 
  \begin{bmatrix}
    \,   \cos  \left( \theta+ \frac{2 \pi}{3} \right) \, \\
    \,   \sin  \left( \theta+ \frac{2 \pi}{3} \right) \, \\
  \, 0 \,
 \end{bmatrix}, \;\; 
  {\mathbf{V}}_{3} = 
   \begin{bmatrix}
   \,   \cos  \left( \theta + \frac{4 \pi}{3} \right) \, \\
    \,   \sin  \left( \theta + \frac{4 \pi}{3} \right) \, \\
  \, 0 \,
 \end{bmatrix}, 
  {\mathbf{V}}_{4} = 
  \begin{bmatrix}
   \, 0  \, \\
   \, 0  \, \\
  \, 1 \,
 \end{bmatrix}, \;\; 
 \end{equation}
 \begin{equation} \label{hCLP points 2}
   {\mathbf{V}}_{5} = 
  \begin{bmatrix}
    \, \cos \left( - \theta \right) \, \\
   \,  \sin \left( - \theta \right)  \, \\
  \, 0 \,
 \end{bmatrix}, \;\; 
  {\mathbf{V}}_{6} = 
  \begin{bmatrix}
    \,   \cos  \left( -\theta+ \frac{2 \pi}{3} \right) \, \\
    \,   \sin  \left( -\theta+ \frac{2 \pi}{3} \right) \, \\
  \, 0 \,
 \end{bmatrix}, \;\; 
  {\mathbf{V}}_{7} = 
   \begin{bmatrix}
   \,   \cos  \left( -\theta + \frac{4 \pi}{3} \right) \, \\
    \,   \sin  \left( -\theta + \frac{4 \pi}{3} \right) \, \\
  \, 0 \,
 \end{bmatrix}, 
  {\mathbf{V}}_{8} = 
  \begin{bmatrix}
   \, 0  \, \\
   \, 0  \, \\
  \, -1 \,
 \end{bmatrix}, \;\; 
  \end{equation}  
as the constant unit vector fields, we find minimal surfaces such that the Chern-Ricci function
\begin{equation} \label{hCLP one}
\sum_{\mathbf{V} \in \left\{ {\mathbf{V}}_{1}, \cdots, {\mathbf{V}}_{8} \right\}} \ln \frac{\, {\left(1 -  {\mathbf{n}}_{{}_{{\mathbf{V}} }}   \right)}^2 \,}{\sqrt{ \, -{\mathcal{K}}} \,}   
 \end{equation}
is constant. It recovers the triply periodic \textbf{hCLP surface}  \cite{EFS2015, FH1992} with 
the Weierstrass data  
\[
 \frac{1}{ {\mathbf{G}} }  {d\mathbf{h}}  = \frac{1}{\, \sqrt{  {\mathbf{G}} {\left(  {\mathbf{G}}^{6} -2 \cos \left(3 \theta \right){\mathbf{G}}^{3}+1 \right)}  }\,} d{\mathbf{G} \, }, \quad 
 \theta \in \left(0, \frac{\pi}{3}\right),
\]
up to associate families. As known in \cite[Theorem 1.2]{EFS2015}, when $\theta \shortarrow{1} \frac{\pi}{3}$, we recover \textbf{Karcher's saddle tower with 6 ends}  \cite{Kar1988}.
\end{example}
  
    \begin{figure}[H]
 \centering
 \includegraphics[height=4.50cm]{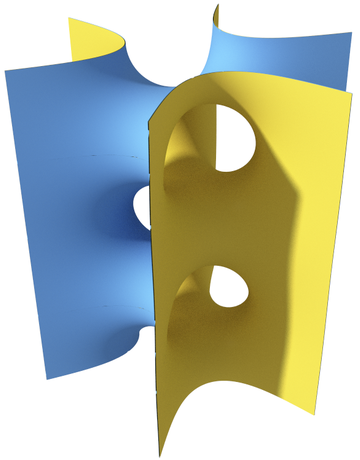}
 \caption{\small{An approximation \cite{WebKsaddle6} of a part of Karcher's saddle tower with 6 ends \cite{Kar1988}}} 
 \end{figure}

\begin{example}[{\textbf{Karcher's saddle tower with 6 ends} as the limit of Schwarz surfaces in \textbf{H family}}] \label{H data} 
Fix the angle $\theta \in \left(0, \frac{\pi}{2}\right)$. Identifying the following eight points in ${\mathbb{S}}^{2}$   
\begin{equation} \label{H points 1}
 {\mathbf{V}}_{1} = 
  \begin{bmatrix}
    \,- \cos \theta  \, \\
   \, 0  \, \\
  \,   \sin \theta \,
 \end{bmatrix}, \;\, 
  {\mathbf{V}}_{2} = 
  \begin{bmatrix}
    \, - \cos \theta  \cos  \left(  \frac{2 \pi}{3} \right) \, \\
    \, - \cos \theta  \sin  \left(  \frac{2 \pi}{3} \right) \, \\
  \,  \sin \theta  \,
 \end{bmatrix}, \;\, 
  {\mathbf{V}}_{3} = 
   \begin{bmatrix}
    \, - \cos \theta  \cos  \left(  \frac{4 \pi}{3} \right) \, \\
    \, - \cos \theta  \sin  \left(  \frac{4 \pi}{3} \right) \, \\
  \,  \sin \theta  \,
 \end{bmatrix}, \;\,
  {\mathbf{V}}_{4} = 
  \begin{bmatrix}
    \, 0  \, \\
   \, 0  \, \\
  \,  1 \,
 \end{bmatrix}, \;\, 
 \end{equation}
 \begin{equation} \label{H points 2}
  {\mathbf{V}}_{5} = 
  \begin{bmatrix}
    \,- \cos \theta  \, \\
   \, 0  \, \\
  \,  - \sin \theta \,
 \end{bmatrix}, \;\,  
  {\mathbf{V}}_{6} = 
  \begin{bmatrix}
    \, - \cos \theta  \cos  \left(  \frac{2 \pi}{3} \right) \, \\
    \, - \cos \theta  \sin  \left(  \frac{2 \pi}{3} \right) \, \\
  \, - \sin \theta  \,
 \end{bmatrix},\;\, 
  {\mathbf{V}}_{7} = 
   \begin{bmatrix}
    \, - \cos \theta  \cos  \left(  \frac{4 \pi}{3} \right) \, \\
    \, - \cos \theta  \sin  \left(  \frac{4 \pi}{3} \right) \, \\
  \,- \sin \theta  \,
 \end{bmatrix}, \;\, 
  {\mathbf{V}}_{8} = 
  \begin{bmatrix}
   \, 0  \, \\
   \, 0  \, \\
  \,  -1 \,
 \end{bmatrix}, \;\;  
  \end{equation} 
  as the constant unit vector fields, we find minimal surfaces such that the Chern-Ricci function
\begin{equation} \label{H one}
\sum_{\mathbf{V} \in \left\{ {\mathbf{V}}_{1}, \cdots, {\mathbf{V}}_{8} \right\}}  \ln \frac{\, {\left(1 -  {\mathbf{n}}_{{}_{{\mathbf{V}} }}   \right)}^2 \,}{\sqrt{ \, -{\mathcal{K}}} \,}   
 \end{equation}
is constant. Take $a=\tan \left( \frac{\pi}{4} - \frac{\theta}{2} \right) \in \left(0, 1\right)$.
It recovers, up to associate families, the  triply periodic minimal surface with 
the Weierstrass data  
\[
 \frac{1}{ {\mathbf{G}} }  {d\mathbf{h}}  = \frac{1}{\, \sqrt{  {\mathbf{G}} {\left(  {\mathbf{G}}^{6}  + \lambda{\mathbf{G}}^{3}+1 \right)}  }\,} d{\mathbf{G} \, }, \quad 
 \lambda =   {a}^{3} + \frac{1}{ {a}^{3} }  >2.
\]
 In terms of the rotated Gauss map $\widetilde{\,\mathbf{G}\,} :=  e^{i \frac{\pi}{3}}  \mathbf{G}$, we have  
 \[
 \frac{1}{ {\widetilde{\,\mathbf{G}\,}} }  {d\mathbf{h}}  = \pm \frac{i}{\, \sqrt{  {\widetilde{\,\mathbf{G}\,}} {\left(  {\widetilde{\,\mathbf{G}\,}}^{6}  - \lambda{\widetilde{\,\mathbf{G}\,}}^{3}+1 \right)}  }\,} d{\widetilde{\,\mathbf{G}\,} \, }=  \pm \frac{i}{\, \sqrt{  {\widetilde{\,\mathbf{G}\,}} {\left(  {\widetilde{\,\mathbf{G}\,}}^{3}  -  
   {a}^{3}   \right)}  \left(  {\widetilde{\,\mathbf{G}\,}}^{3}  -    \frac{1} {{a}^{3}}   \right) \,  }\,} d{\widetilde{\,\mathbf{G}\,} \, },
\]
which induces the Weierstrass data of \textbf{Schwarz H surfaces} \cite{Schwarz1890}. See also \cite[Example 2.3]{EFS2015}. As known in \cite[Theorem 1.2]{EFS2015}, when $a \shortarrow{1} 1$, it induces \textbf{Karcher's saddle tower with 6 ends}  \cite{Kar1988}.
\end{example}
  
  \begin{enumerate}
  \item[] 
  \end{enumerate}
  
 \noindent \textbf{Acknowledgement}. Part of this work was carried out while the author was visiting Jaigyoung Choe at Korea Institute for Advanced Study. 
 He would like to thank KIAS for its  hospitality.

    \begin{enumerate}
  \item[] 
  \end{enumerate}

\bigskip

\end{document}